\documentclass[12pt]{amsart}
\usepackage[T1]{fontenc}
\usepackage{graphics}
\usepackage{setspace}
\usepackage{amsfonts,amssymb,color}
\usepackage[mathscr]{eucal}
\usepackage{amsmath, amsthm}
\usepackage{mathrsfs}
\usepackage{amsbsy}
\usepackage{dsfont}
\usepackage{bbm}
\usepackage{wasysym}
\usepackage{stmaryrd}
\usepackage{url}

\input xypic
\xyoption{all}

\makeindex
\makeglossary

\begin{document}
\baselineskip = 16pt

\newcommand \ZZ {{\mathbb Z}}
\newcommand \NN {{\mathbb N}}
\newcommand \RR {{\mathbb R}}
\newcommand \PR {{\mathbb P}}
\newcommand \AF {{\mathbb A}}
\newcommand \GG {{\mathbb G}}
\newcommand \QQ {{\mathbb Q}}
\newcommand \CC {{\mathbb C}}
\newcommand \bcA {{\mathscr A}}
\newcommand \bcC {{\mathscr C}}
\newcommand \bcD {{\mathscr D}}
\newcommand \bcF {{\mathscr F}}
\newcommand \bcG {{\mathscr G}}
\newcommand \bcH {{\mathscr H}}
\newcommand \bcM {{\mathscr M}}
\newcommand \bcI {{\mathscr I}}
\newcommand \bcJ {{\mathscr J}}
\newcommand \bcK {{\mathscr K}}
\newcommand \bcL {{\mathscr L}}
\newcommand \bcO {{\mathscr O}}
\newcommand \bcP {{\mathscr P}}
\newcommand \bcQ {{\mathscr Q}}
\newcommand \bcR {{\mathscr R}}
\newcommand \bcS {{\mathscr S}}
\newcommand \bcV {{\mathscr V}}
\newcommand \bcU {{\mathscr U}}
\newcommand \bcW {{\mathscr W}}
\newcommand \bcX {{\mathscr X}}
\newcommand \bcY {{\mathscr Y}}
\newcommand \bcZ {{\mathscr Z}}
\newcommand \goa {{\mathfrak a}}
\newcommand \gob {{\mathfrak b}}
\newcommand \goc {{\mathfrak c}}
\newcommand \gom {{\mathfrak m}}
\newcommand \gon {{\mathfrak n}}
\newcommand \gop {{\mathfrak p}}
\newcommand \goq {{\mathfrak q}}
\newcommand \goQ {{\mathfrak Q}}
\newcommand \goP {{\mathfrak P}}
\newcommand \goM {{\mathfrak M}}
\newcommand \goN {{\mathfrak N}}
\newcommand \uno {{\mathbbm 1}}
\newcommand \Le {{\mathbbm L}}
\newcommand \Spec {{\rm {Spec}}}
\newcommand \Gr {{\rm {Gr}}}
\newcommand \Pic {{\rm {Pic}}}
\newcommand \Jac {{{J}}}
\newcommand \Alb {{\rm {Alb}}}
\newcommand \Corr {{Corr}}
\newcommand \Chow {{\mathscr C}}
\newcommand \Sym {{\rm {Sym}}}
\newcommand \Prym {{\rm {Prym}}}
\newcommand \cha {{\rm {char}}}
\newcommand \eff {{\rm {eff}}}
\newcommand \tr {{\rm {tr}}}
\newcommand \Tr {{\rm {Tr}}}
\newcommand \pr {{\rm {pr}}}
\newcommand \ev {{\it {ev}}}
\newcommand \cl {{\rm {cl}}}
\newcommand \interior {{\rm {Int}}}
\newcommand \sep {{\rm {sep}}}
\newcommand \td {{\rm {tdeg}}}
\newcommand \alg {{\rm {alg}}}
\newcommand \im {{\rm im}}
\newcommand \gr {{\rm {gr}}}
\newcommand \op {{\rm op}}
\newcommand \Hom {{\rm Hom}}
\newcommand \Hilb {{\rm Hilb}}
\newcommand \Sch {{\mathscr S\! }{\it ch}}
\newcommand \cHilb {{\mathscr H\! }{\it ilb}}
\newcommand \cHom {{\mathscr H\! }{\it om}}
\newcommand \colim {{{\rm colim}\, }} % colimit
\newcommand \End {{\rm {End}}}
\newcommand \coker {{\rm {coker}}}
\newcommand \id {{\rm {id}}}
\newcommand \van {{\rm {van}}}
\newcommand \spc {{\rm {sp}}}
\newcommand \Ob {{\rm Ob}}
\newcommand \Aut {{\rm Aut}}
\newcommand \cor {{\rm {cor}}}
\newcommand \Cor {{\it {Corr}}}
\newcommand \res {{\rm {res}}}
\newcommand \red {{\rm{red}}}
\newcommand \Gal {{\rm {Gal}}}
\newcommand \PGL {{\rm {PGL}}}
\newcommand \Bl {{\rm {Bl}}}
\newcommand \Sing {{\rm {Sing}}}
\newcommand \spn {{\rm {span}}}
\newcommand \Nm {{\rm {Nm}}}
\newcommand \inv {{\rm {inv}}}
\newcommand \codim {{\rm {codim}}}
\newcommand \Div{{\rm{Div}}}
\newcommand \CH{{\rm{CH}}}
\newcommand \sg {{\Sigma }}
\newcommand \DM {{\sf DM}}
\newcommand \Gm {{{\mathbb G}_{\rm m}}}
\newcommand \tame {\rm {tame }}
\newcommand \znak {{\natural }}
\newcommand \lra {\longrightarrow}
\newcommand \hra {\hookrightarrow}
\newcommand \rra {\rightrightarrows}
\newcommand \ord {{\rm {ord}}}
\newcommand \Rat {{\mathscr Rat}}
\newcommand \rd {{\rm {red}}}
\newcommand \bSpec {{\bf {Spec}}}
\newcommand \Proj {{\rm {Proj}}}
\newcommand \pdiv {{\rm {div}}}
\newcommand \wt {\widetilde }
\newcommand \ac {\acute }
\newcommand \ch {\check }
\newcommand \ol {\overline }
\newcommand \Th {\Theta}
\newcommand \cAb {{\mathscr A\! }{\it b}}

\newenvironment{pf}{\par\noindent{\em Proof}.}{\hfill\framebox(6,6)
\par\medskip}

\newtheorem{theorem}[subsection]{Theorem}
\newtheorem{conjecture}[subsection]{Conjecture}
\newtheorem{proposition}[subsection]{Proposition}
\newtheorem{lemma}[subsection]{Lemma}
\newtheorem{remark}[subsection]{Remark}
\newtheorem{remarks}[subsection]{Remarks}
\newtheorem{definition}[subsection]{Definition}
\newtheorem{corollary}[subsection]{Corollary}
\newtheorem{example}[subsection]{Example}
\newtheorem{examples}[subsection]{examples}
\title[Bloch's conjecture]{Bloch's conjecture for fake projective planes with an automorphism of order 3}

\author{Kalyan Banerjee}

\address{SRM University AP, India}

\email{kalyan.ba@srmap.edu.in}

\begin{abstract}
In this short note, we prove that an involution on certain examples of surfaces of general type with $p_g=0=q, K^2=3$, acts as identity on the Chow group of zero cycles of the relevant surface. In particular, we consider examples of such surfaces when the quotient is birational to an Enriques surface or to a surface of Kodaira dimension one and show that the Bloch conjecture holds for such surfaces. As a consequence, we show that the Bloch conjecture holds for certain examples of fake projective planes.
\end{abstract}
\maketitle

\section{Introduction}
In \cite{M}, Mumford  proved that if the geometric genus of a smooth projective complex surface is greater than zero, then the albanese kernel cannot be parametrized by an algebraic variety. In \cite{B1}, Bloch conjectured the following strong converse of the Mumford's theorem: if the geometric genus of a smooth projective complex surface is zero then the albanese kernel is zero. It is known for surfaces not of general type with geometric genus equal to zero due to \cite{BKL}. After that, the conjecture was verified for some examples of surfaces of general type with geometric genus zero due to \cite{B},\cite{IM}, \cite{PW}, %\cite{Gul1}, %\cite{Gul2}, 
\cite{V}, \cite{VC}. In a previous paper, the author solved the Bloch conjecture for surfaces of general type with an involution and $p_g=q=0$, $K^2=1,2$ such that the quotient of the surface by the involution is birational to an Enriques surface. This is the main theorem in \cite{Ba}. In this context, there is a recent preprint, \cite{Ba1}, which claims the Bloch conjecture for all fake projective planes, but the method present there is different and does not use automorphisms.

 The aim of this manuscript is to verify Bloch's conjecture for some examples of surfaces of general type with geometric genus zero. Namely, surfaces with an involution with $K^2=3$ such that the quotient of this surface by the involution is birational to an Enriques surface or to a surface of Kodaira dimension one. Such surfaces are classified in \cite{Ri}. It is shown there that if $S/i$ is not rational, then it is birational to an Enriques surface or it has the Kodaira dimension one. The possibilities of the ramification locus are also discussed in \cite{Ri}.

 The novelty of this paper is that the calculations related to the ramification locus are different from those present in \cite{Ba} and also that there is a new case involved here when the quotient surface is of Kodaira dimension one. Using this, we produce certain examples of fake projective planes for which the Bloch conjecture holds.

 According to \citation[Section 5.2]{Ri} we can construct examples of surfaces with $K^2=3$ starting from the bi-double cover constructions of a rational surface. 

 Let $X$ be a fake projective plane such that there is an action of $ \ZZ/3\ZZ$. Let $Y$ be the unique (up to isomorphism) minimal model of the quotient $X/(\ZZ/3)$. This is a surface of general type with $K^2=3, p_g=0$.

 With this information, we have the following theorem:

\begin{theorem}
Let $X$ be a fake projective plane and $\sigma$ be an automorphism of order three on $X$. Let  $Y$ be the minimal surface  of general type with $p_g=0,K^2=3$ birational to $X/\sigma$. Then $\sigma$ acts as the identity of $CH_0(X)_0$ (the degree zero cycles modulo rational equivalence).
\end{theorem}

The proof of this theorem is the proof of Theorem 2.5 in section 2.

The examples of surfaces satisfying the Theorem above are studied in the paper by Cartwright-Stegar \cite{CS}. In Table 2 present in \cite{CS} the examples that satisfy the theorem above are exactly the first case of $\mathcal C_2$ present in this table. There are exactly six fake projective planes up to isomorphism admitting an automorphism of order 3 and satisfying the condition of the theorem.
Why they satisfy the condition of the theorem is discussed on paragraph below Theorem 2.2 in \cite{CS}.

Now, with the above result, if we have $\CH_0(Y)_0$ satisfies Bloch's conjecture, then we have Bloch's conjecture for $\CH_0(X)_0$.

%Consider a surface $S$ of general type with geometric genus equal to zero with $K^2=3$. Suppose such a surface has an involution $i$, such that the bicanonical map is not composed with the involution and the quotient is not rational, then by Theorem 5 of \cite{Ri}, we have the classification for the branch locus of the quotient map. Consider the case when this branch locus contains $5$ curves of self intersection $-1$ and an elliptic curve (in such a case the quotient is a surface not of general type with $p_g=0$). Then taking a very ample divisor on $S/i$ and arguing as in \ref{theorem2} we have:

%\begin{theorem}
%\label{theorem3}
%The Bloch's conjecture holds for the surface $S$ with $K_S^2=3$ and having an involution such that the bicanonical map is not composed with an involution, with the components of the branch locus mentioned above.
%\end{theorem}

%As an application we show that the Bloch conjecture holds true on certain fake projective planes.

The technique used to prove these results is on the same line as in \cite{Voi},  where the conjecture is verified for a K3 surface with a symplectic involution and in \cite{Ba} for a surface of general type with $p_g=q=0$ and quotient by the involution being Enriques. Voisin invokes the notion of finite dimensionality in the Roitman sense \cite{R1} and demonstrates that the correspondence given by the difference of the graph of the involution and the diagonal is finite dimensional.

{\small \textbf{Acknowledgements:} The author thanks SRM AP for hosting this project. The author thanks Vladimir Guletskii for many useful conversations relevant to the theme of the paper. The author thanks Claire Voisin for her advice on the theme of the paper. The author thanks Claudio Pedrini and Charles Weibel for discussions related to the topic of the paper. Finally the author thanks the referee for careful reading and suggestions to improve the manuscript.}

\section{Finite dimensionality in the sense of Roitman and $\CH_0$}

First, we recall the notion of finite dimensionality in the sense of Roitman \cite{R1}. Let $X$ be a smooth projective variety over $\CC$ and let $P$ be a subgroup of $\CH_0(X)$, then we will say that $P$ is finite dimensional in the Roitman sense, if there exists a smooth projective variety $W$ and a correspondence $\Gamma$ on $W\times X$, such that $P$ is contained in the set $\{\Gamma_*(w)|w\in W\}$.

The following proposition is due to Roitman and also to Voisin. For convenience, we recall the theorem.

Let $M,X$ be smooth projective varieties of dimension $d$ and $Z$ be a correspondence between $M$ and $X$ of codimension $d$.

\begin{theorem}\cite{Voi}[Theorem 2.3]
\label{theorem 1}
Assume that the image of $Z_*$ from $\CH_0(M)$ to $\CH_0(X)$ is finite-dimensional in the Roitman sense. Then the map $Z_*$ from ${\CH_0(M)}_{hom}$ to $\CH_0(X)$ factors through the albanese map from ${\CH_0(M)}_{hom}$ to $\Alb(M)$.
\end{theorem}

This theorem is important because we apply to the case $M=X=S$ with zero irregularity and $Z=\Delta-Gr(i)$ being the difference of the diagonal and the graph of the involution. Then it will follow that the involution acts as an identity on $CH_0(S)_{hom}$ provided that $Z_*$ has a finite-dimensional image in the Roitman sense, since Albanese is trivial in this case.

By the classification of smooth minimal surfaces with $K^2=3$ and the bicanonical map not composed with the involution, we have the following possibilities for the branch locus of the involution as per \cite{Ri}[Theorem 5], $B=B'+\sum_i A_i$ where $A_i$'s are finitely many $-2$ curves and $B'$ is :

1. It is a smooth irreducible curve of self-intersection $-6$ and of genus zero.

2. It is a union of two smooth irreducible curves of self-intersection $-6,-4$ and of genus zero.

3. It is a union of two smooth irreducible curves of self-intersection $-2$ and genus one, self-intersection $-4$ and genus zero, respectively.

4. It is a union of three smooth irreducible curves, one of self-intersection $-2$ and genus one, and two of self-intersection $-4$ and of genus zero, respectively.

5. It is a smooth curve of genus one and self-intersection $-2$.

We consider the first four cases where the quotient is birational to an Enriques surface.

This is why we prove the following theorem with the requirement prescribed above on the branch locus. This theorem follows the line of approach of Proposition 2.5 in \cite{Voi}.

\begin{theorem}
\label{theorem2}
Let $S$ be a smooth minimal surface of general type with $p_g=0=q$, let $i:S\to S$ be an involution, and $K_S^2=3$. Let $S/i$ be isomorphic to an Enriques surface such that the branch locus is one of the first four possibilities mentioned above. Then the anti-invariant part
$$\CH_0(S)^-=\{z\in \CH_0(S):i_*(z)=-z\}$$
is finite-dimensional in the Roitman sense.
\end{theorem}
\begin{proof}
Let $S$ be a surface of general type with $p_g=q=0$ and $K_S^2=3$ such that the quotient surface $S/i$ is birational to an Enriques surface. Then the ramification locus of the involution is, as mentioned above (the first four possibilities).

Take $L$ a very ample line bundle on $S'=S/i$ and we know that  $2K_{S'}=0$ in the case the surface is an Enriques surface. By the adjunction formula we have that
$$L^2+L.K=2g-2$$
where $g$ is the genus of a smooth projective curve in the linear system of $L$. Since $L$ is very ample, the left-hand side of the above equality is positive, so the genus of the curves in the linear system of $L$ is positive and greater than one.  

Now we have the exact sequence
$$\bcI_C\to \bcO(S')\to \bcO(S')/\bcI_C\to 0$$
where $C$ is a smooth projective curve in the linear system of $L$, $\bcI_C=\bcO(-C)$. 
Now, tensor the above sequence with $\bcO(C)$
we get
$$0\to \bcO(S')\to \bcO(C)\to \bcO_C(C)\to 0$$
which gives
$$0\to \CC\to H^0(S',L)\to H^0(C,L|_C)\to 0$$
by the adjunction formula and using the fact that the irregularity of the surface $S'$ is zero.
Now, let us calculate the dimension of $H^0(C,L|_C)$. By Riemann-Roch for smooth projective curves, the above is 
$$\deg(L|_C)-g+1+\dim (H^0(C, K_C-L|_C))$$
The degree of $L|_C$ is $2g-2$. Since the degree of $K_C-L|_C=0$ we have $\dim(H^0(C, K_C-L|_C))=0$ or $\dim(H^0(C, K_C-L|_C))=1$ and the latter occurs if $K_C-L|_C=0$, but since  $2(K_C-L|_C)=0$ we have  $\dim(H^0(K_C-L|_C))=0$. Therefore, the dimension of $H^0(C,L|_C)=g-1$ and hence $\dim(H^0(S',L))=g$. Hence, the dimension of the linear system $|L|$ is $g-1$.

Now for $C$ in the linear system $L$, let  $\wt{C}$  be the pre-image of $C$ in $S$. It is a ramified double cover of $C$. Consider the general $C$ and the universal family $\mathcal {C}$ consisting of curves in the linear system $L$ on $S'$. Let $\mathcal {\wt{C}}$ be the family of ramified double covers on $S$ corresponding to $\mathcal C$. This family is a family of curves on $S$. According to Theorem 2 on Page 141 \cite{Sha}, a general curve $\wt{C}$ is non-singular on $S$. Since $\wt{C}\to C$ there is a 2:1 map , there are two components of $\wt{C}$, which are denoted as $\wt{C_1}, \wt{C_2}$ and since $\wt{C}$ is non-singular $\wt{C_1}.\wt{C_2}=0$.
Let us give details on this. Since $\wt{C}$ is ample , we have $\wt{C}.B>0$, here $B$ is the fixed locus of the involution. Since $$\wt{C}=\wt{C_1}\cup \wt{C_2}$$
we have 
$$\wt{C}.B=\wt{C_1}.B+\wt{C_2}.B$$
Also, observe that $B$ consists of all points with the property that $i(x)=x$. Along with this property and the fact that there is an induced involution on $\wt{C}$, the intersection number 
$$\wt{C_1}.B+\wt{C_2}.B=|\wt{C_1}\cap \wt{C_2}|$$
for general $C$. Since the above number is positive, the cardinality $|\wt{C_1}\cap \wt{C_2}|$ is positive. Hence, along the intersection points the curve $\wt{C}$ is not smooth, contradicting its non-singularity for a general $C$. Hence, $\wt{C}$ must be connected.

Let $\Gamma$ be the correspondence on $S\times S$ given by $\Delta_S-Gr(i)$. We now prove that $$\Gamma_*(S^{g-1})=\Gamma_*(S^{g-2})\;.$$
Here $\Gamma_*(S^g)$ is the subset of $\CH_0(S)$ consisting of cycles of the form 
$$\sum_{j=1}^g [s_j-i(s_j)]\;.$$
Let $s=(s_1,\cdots,s_{g-1})$ be a general point of $S^{g-1}$ and let $\sigma_j$ be the image of $s_j$ under the quotient map $S\to S'$. Then a generic $(s_1,\cdots,s_{g-1})$ gives rise to a generic $(\sigma_1,\cdots,\sigma_{g-1})$ in $S'^{g-1}$ and there exists a unique $C_s$ in the linear system $|L|$ such that it contains all the $\sigma_j$. The curve $C_s$ is general in the linear system, so its inverse image under the quotient map is a ramified double cover $\wt{C_s} $ of $C_s$, where $\wt{C_s}$ contains all the points $s_j$. The zero cycle
$$z_s=\sum_l [s_l ]-\sum_l [i(s_l)]$$
is supported on $J(\wt{C_s})$ and is antiinvariant under $i$, so it is actually in the Prym variety $P(\wt{C_s}/C_s)$ of the ramified double cover $\wt{C_s}\to C_s$, which is a $g-1+\deg(B.C)/2$ dimensional abelian variety and also is the kernel of the norm map from $J(\wt{C_s})$ to $J(C_s)$. So we have the following set-theoretic map
$$(s_1,\cdots,s_{g-1})\mapsto z_s\;,$$
from $U\to \CH_0(S)^{-}$
where $U$ is the subset of $S^{g-1}$ consisting of points $(s_1,\cdots,s_{g-1})$ such that $s_i\neq s_j$ for $i\neq j$.
This map factors through a morphism
$$f:U\to \bcP(\wt{\bcC}/\bcC) $$
This is because we can take a point $(s_1,s_2,\cdots,s_{g-1})$ in $U$  and there is a unique curve $C_s$ whose ramified double cover $\wt{C_s}$ contains all these points. Then map 
$$(s_1,\cdots,s_{g-1})\to \sum_{j=1}^{g-1}(s_j-i(s_j))$$
in $P(\wt{C_s}/C_s)$ and further push-it forward to $\CH_0(S)$, induced by closed embedding $C_s\to S$. Then we have 
$$(s_1,\cdots,s_{g-1})\to \sum_{j=1}^{g-1}(s_j-i(s_j))\to \sum_{j=1}^{g-1}[s_j-i(s_j)]\;.$$

Case I (When $B'$ is of genus 0):

Here $\bcC\to |L|_0$ is the universal smooth curve over the Zariski open set $|L|_0$ which is of dimension $g-1$ and $\wt{\bcC}\to |L|_0$ is its universal double cover, and $\bcP(\wt{\bcC}/\bcC)$ is the corresponding Prym fibration. So, the Prym fibration is of dimension $2g-2+\deg(B.C)/2$. Since all components are $-2$ curves of the ramification locus, we have $A_i.C=0$. We have $B'$ a possibly reducible rational curve. 

We have to prove that the morphism $f$ has positive dimensional fibers. Consider the subsets
$$S_{B_i'}=\{(s_1,\cdots,s_{g-1})|\textit{exactly $i$ of $s_1,\cdots,s_{g-1}$ are in $B'$}\}$$

Then we can consider the map from $S_{D_i}\cong S^{g-i-1}\times B'^i$ to $\bcP(\wt{\bcC}/\bcC)$ (as $B'$ is rational, the corresponding Jacobian of its smooth components is trivial).

Given a general point $(s_1,\cdots,s_{g-1})$ the image of it is contained in $P(\wt{C_s}/C_s)$. Now, the $s_i$'s that are not in $B'$, they are in the image of the map $\Sym^{g-i-1}\wt{C_s}\to J(\wt{C_s})$. So, it is contained in a subvariety in $J(\wt{C_s})$ as well as in $P(\wt{C_s}/C_s)$ that is of dimension $g-1+\deg(B'.C_s)/2$. So, the dimension of the image is less than $g-i-1$ and $g-1+\deg(B'.C_s)/2$. In particular, the dimension of the image is less than $g-i-1$.  So,the image of $\wt{C_s}^{g-i-1}\times B'^i$ has at most a dimension of $g-i-1$, as it is contained in $P(\wt{C_s}/C_s)$. Therefore, the image of the map from $S^{g-i}\times B'^i$ is in the subvariety of  $\bcP(\wt{\bcC}/\bcC)$ and is of dimension at most $g-i-1+g-1=2g-i-2$ (adding the fiber dimension with the base dimension). Now we can stratify the general points of $S^{g-1}$ as a point in $S^{g-i-1}\times B'^i$ for some $i$ (ranging from $0$ to $g$). So, the dimension of the fiber of $S^{g-1}\to \bcP(\wt{\bcC}/\bcC)$ (which is the image of the map $S^{g-i-1}\times B'^i$) is at least
$$\max\{2g-2-2g+i+2=i\}\;.$$
Now, given a curve $C_s$, the intersection number $B'.C_s\geq 1$. So, taking a large power of the line bundle $L$, we have $i>1$. 

Therefore, the fiber above the map $S^g\to \bcP(\wt{\bcC}/\bcC)$ is always positive. So, it contains a curve. Say $F_s$.

So, let $z_s$ be the zero cycles as above, supported by the Prym variety $P(\wt{C_s}/C_s)$. Then there exists a curve $F_s$ in $U$ such that for any $(t_1,\cdots,t_{g-1})$ in $F_s$, the cycle $z_t$ is rationally equivalent to $z_s$.  Now $S$ is a minimal surface of general type with $p_g=0$, so we have $5K_S$ is very ample. So, the divisor $\sum_l pr_l^{-1}(5K_S)$ is ample on $S^{g-1}$, call it $D$. Then $D$ meets $F_s$, for all $z_s$. We have the zero cycle $z_s$ supported on $S^{g-k-1}\times C^k$, where $C$ is a smooth curve in the linear system of $5K_S$.

In fact, consider the divisor
$\sum_i \pi_i^{-1}(5K_{S})$ on the product $S^{g-1}$. This divisor is ample, so it intersects $F_s$, so we find that there exist points in $F_s$  contained in $C\times S^{g-2}$ where $C$ is in the linear system of $5K_{S}$. Then consider the elements of $F_s$ of the form $(c,s_1,\cdots,s_{g-2})$, where $c$ belongs to $C$. Taking into account the map from $S^{g-2}$ to $A_0(S)$ given by
$$(s_1,\cdots,s_{g-2})\mapsto \sum_j s_j+c-\sum_j i(s_j)-i(c)\;,$$
we see that this map factors through the product of the Prym fibration  and the map from $S^{g-2}$ to $\bcP(\wt{\bcC}/\bcC)$ has positive dimensional fibers, since $i$ is large (here $i$ is the number of co-ordinates which contain points of $B'.C_s$). So, it means that, if we consider an element $(c,s_1,\cdots,s_{g-2})$ in $F_s$ and a curve through it, then it intersects the ample divisor given by $\sum_i \pi_i^{-1}(5K_{S})$, in $S^{g-2}$. Then we have some of $s_i$ contained in $C$. So, iterating this process, we find that there exist elements of $F_s$ that are supported on $C^k\times S^{g-k-1}$, where $k$ is some natural number depending on $g$. The genus of the curve $C$ is $46$ depending on $K_S^2=3$. We can choose the very ample line bundle $L$ such that $g$ is very large, so $k$ is larger than the genus of $C$. So we have $z_s$ supported on $C^{46}\times S^{g-k-1}$, hence we have $\Gamma_*(U)=\Gamma_*(S^{i_0})$, where $i_0=45+g-k$  which is strictly less than $g-1$, since the genus of $C$ is strictly less than $k$. That is, we have shown that for any $(s_1,\cdots,s_{g-1})$ in $U$ in $S^{g-1}$, $z_s$ is rationally equivalent to a cycle on $S^{i_0}$. Using the fact that $\Gamma_*(U)$ is $\Gamma_*(S^{g-1})$, we have shown that $\Gamma_*(S^{g-1})=\Gamma_*(S^{i_0})$.

Case II (When $B'$ is of genus 1): In the last three cases, there is a presence of an elliptic curve in the ramification locus along with (possibly) some rational curves. So in this case the map from $S^{g-1}\to \CH_0(S)^{-}$ factors through $\bcP(\wt{C}/C)\times E$, here $E$ is the elliptic curve in the ramification locus. In this case the image of $S^{g-1}$ in $\bcP(\wt{C}/C)\times E$ is $S^{g-1-i}\times E$ which is of dimension $2g-i-1$, and the fiber is of dimension $2g-2-(2g-i-1)=i-1$ and since $E.C_s>0$, we have $i>1$ and we can make $i$ very large taking large powers of the line bundle $L$. Then the same argument follows.

Now by induction $\Gamma_*(S^i)=\Gamma_*(S^m)$ follows \cite{Voi}.

\end{proof}

\begin{theorem}

Let $S$ be a smooth minimal surface of general type with $p_g=0=q$, let $i:S\to S$ be an involution, and $K_S^2=3$. Let $S/i$ be birational to a surface of Kodaira dimension 1 such that the branch locus is the fifth possibility mentioned above. Then the anti-invariant part
$$\CH_0(S)^-=\{z\in \CH_0(S):i_*(z)=-z\}$$
is finite-dimensional in the Roitman sense.
\end{theorem}

\begin{proof}
Let $S'$ be the minimal desingularization of the quotient surface $S/i$. According to the assumption of the theorem $S'$ is of Kodaira dimension $1$ and hence has an elliptic surface with geometric genus zero. The canonical bundle of the elliptic surface is 
$$K_{S'}=p^*(K_B+D)$$
where $D$ is divisor at $B$ the base of the elliptic surface $p:S'\to B$. The linear system of $D$ is base point free. Let $L$ be a very ample line bundle on $S'$. By the adjunction formula we have 
$$L^2+L.K_{S'}=2g-2$$
Now $$L.K_{S'}=L.p^*(K_B+D)$$
Here $L.p^*(D)>0$ as $D$ is an effective divisor. Since the irregularity of $S'$ is zero, suppose that $p$ has a section in that case $B$ is $\PR^1$. Then $K_{\PR^1}=-2P$ for a point $P$ in $B$. By construction of elliptic surfaces $deg(D)>2$ we have $$L. p^*(K_B+D)>0$$ and we say it is equal to $n$. Then by Riemann-Roch 
$H^0(C,L|_C)\geq 2g-2-g+1-n=g-1-n$. So, the dimension of $H^0(S',L)\geq g-n$. Let this dimension be $g-n+m$, here $m\geq 0$.
Then we can consider $S^{g-n+m}$, and any point in general position (coordinately distinct) on $S^{g-n+m}$ is on a unique curve $\wt{C}$, such that $C\in |L|$. Then according to the possibility of the ramification locus, the map $\wt{C}\to C$ is branched along $C.E$, where $E$ is an elliptic curve on $S'$ with self-intersection $-2$. 

We claim that the map $S^{g-n+m}\to \CH_0(S)^-$ factors through the
$\bcP(\wt{C}/C)\times E$ as in the previous theorem. The image of this map has dimension $2g-2n+2m-i+1$, where $i$ is the number of points on $C.E$. We can make this $i$ large by taking a large power of the line bundle $L$. Then the fibers of the map 
$$S^{g-n+m}\to \bcP(\wt{C}/C)\times E$$ are of dimension $i-1$. Since $i>>1$, we have that the fiber is of positive dimension.

Then we argue as in the previous theorem \ref{theorem2}.
\end{proof}

By the classification of surfaces of general type with $p_g=q=0$ and $K^2=3$ with an involution (such that the bicanonical map is not composed with the involution) according to \cite{Ri}, we see that the quotient surface is birational to an Enriques surface, or it is a rational surface, or it is of Kodaira dimension one. By applying the above theorem, we get:

\begin{theorem}
For all surfaces of general type with $p_g=q=0$ and $K^2=3$ with an involution such that the bicanonical map is not composed with the involution and the quotient is birational to an Enriques surface or to a surface of Kodaira dimension one,  the involution acts as the identity on the group of algebraically trivial zero cycles modulo rational equivalence.
\end{theorem}

According to \cite{Ri} [Section 5.2] we can construct examples of such surfaces starting from the bi-double cover constructions of a rational surface.

Let $X$ be  a fake projective plane such that there is an action of the cyclic group  $<\sigma>\cong\ZZ/3\ZZ$ and the quotient is a surface birational to a minimal surface of general type with $p_g=0, K^2=3$. We denote this surface by $Y$. Then there is a map from $\wt{X}\to Y$, where $\wt{X}$ is the resolution of singularities of the component  of  $X\times_{X/\sigma} Y$, such that the rational map from $X$ to the above component  is dominant.

\textit{Proof of theorem 1.1}:

\begin{theorem}\label{Thm 2.5}
Let $X$ be a fake projective plane with an automorphism $\sigma$  of order three on $X$. Let $\wt{X}$ be as above. Let $Y$ be the quotient of $\wt{X}$ by $\sigma'$, where $\sigma'$ is the automorphism induced by the action of $\sigma$ on $X$.  Suppose that the Bloch conjecture holds on $Y$, then $\sigma$ acts as the identity of $CH_0(X)_0$ (the degree zero cycles modulo rational equivalence). Consequently the Bloch conjecture holds for $X$.
\end{theorem}

\begin{proof}
Consider the surface $Y$, such that it is the quotient of $\wt{X}$ by an order three automorphism. Note that $\sigma$ induces an order three automorphism in $\wt{X}$, let us also denote this by $\sigma$. Let $L$ be the very ample line bundle $nK_Y$ where $n>5$. Then we have, by the adjunction formula, $$ n (n + 1) K_Y^2 = 2g-2$$
where $g$ is the genus of the curve in the linear system $|nK_Y|$.
Now $K_Y^2=3$, so we have 
$$3n(n+1)=2g-2\;.$$
We have, $|nK_Y|=\frac{3n(n-1)}{2}=m$. Now, given a point $(x_1,\cdots,x_m)$ on $\wt{X}^m$, which is in general position  such that $x_i\neq x_j$, there is a unique curve $\wt{C_t}$ , where $C_t\in |nK_Y|$. This $\wt{C_t}\to C_t$ is in general a \'etale triple cover of $C_t$ as the fixed locus $X\to X/\sigma$ is just a finite set of points. For a general $C_t$, $\wt{C_t}$ is connected and smooth by the extended Bertini's theorem as in \citation [Theorem 2, page 141]{Sha}. We have a push-forward map from $J(\wt{C_t})$ to $J(C_t)$. So, given the cycle 
$$\sum_{i=1}^m x_i-\sigma(x_i)$$
there is a unique curve $\wt{C_t}$, such that $J(\wt{C_t})$ contains this zero cycle. Now, for this zero cycle, if we apply $\sigma$ on it we have 
$$\sum_{i=1}^m \sigma(x_i)-\sigma^2(x_i)$$
if this is equal to 
$$\sum_{i=1}^m x_i-\sigma(x_i)$$
then 
$$\sum_{i=1}^m x_i+\sigma^2(x_i)=2\sigma(x_i)$$
Now, we have 
$$\sum_{i=1}^m x_i+\sigma(x_i)+\sigma^2(x_i)=0$$
 by the Bloch conjecture on $Y$. Thus we have 
$$3\sum_{i=1}^m\sigma(x_i)=0\;.$$
This means that the effective zero cycle $\sum_{i=1}^m x_i$ is rationally equivalent to zero, which is not possible by degree argument. So, the zero cycle 
$$\sum_{i=1}^m x_i-\sigma(x_i)$$ is not $\sigma$ invariant and hence is in the kernel of the map 
$$\ker(J(\wt{C_t})\to J(C_t)):= P_t$$
The dimension of $P_t$ is 
$$\dim(J(\wt{C_t}))-\dim(J(C_t))=2g-2$$
by the Riemann-Hurwitz formula. So, the map from $S^g$ to $\CH_0(S)_0$ factors through the map from $X^m\to \bcP$, where $\bcP$ is the fibration over $|nK_Y|$ having general fiber $P_t$. The dimension of the image of $X^m$ in $\bcP$ is of dimension 
$$2\dim(|nK_Y|)=3n(n-1)/2$$
Therefore, the map $X^m\to \bcP$ is generically  finite. However, if we consider the divisor $5K_Y$, it is very ample and hence intersects $nK_Y$ at $15n$ points. Therefore, of the $m$ points $15n$ the points are supported in $\wt{D}^k$ for some positive integer $k$, where $\wt{D}$ is the double cover of a curve $D$ is in the linear system of $|5K_Y|$. So, the map of $X^m\to \CH_0(X)_0$ actually factors through $\bcP\times J(\wt{D})$. The dimension of the image is hence $$m-15n+m+\dim(J(\wt{D}))=2m-15n+90$$
and the dimension of the fiber is 
$$15n-90>0$$
if $n>6$. Taking $n$ very large, we have $n>>6$. Since the fiber dimension is positive, we have the pullback of $5K_{\wt{X}}$ which is ample intersects this fiber, and arguing as in \ref{theorem2} we find that $\CH_0(\wt{X})_0$ is finite dimensional.
\end{proof}

\begin{example}
Examples of fake projective planes $X$ admitting an order 3 automorphism such that the quotient is a surface, birational to a surface with $p_g=0,K^2=3$ are studied by \cite{CS}, \cite{K}, \cite{PY} and are called Cartwright-Steger surfaces. These examples  are mentioned in Table 2 in \cite{CS} in the first sub-case of $\mathcal C_2$ and there are exactly $6$ isomorphism classes of such surfaces. 
\end{example}

\end{document}